\documentclass{article}
\usepackage{amssymb}
\usepackage{amsmath}

\newtheorem{theorem}{Theorem}

\newtheorem{application}[theorem]{Application}

\newtheorem{condition}[theorem]{Condition}

\newtheorem{corollary}[theorem]{Corollary}
\newtheorem{criterion}[theorem]{Criterion}
\newtheorem{definition}[theorem]{Definition}
\newtheorem{example}[theorem]{Example}

\newtheorem{lemma}[theorem]{Lemma}
\newtheorem{notation}[theorem]{Notation}

\newtheorem{proposition}[theorem]{Proposition}
\newtheorem{remark}[theorem]{Remark}

\newenvironment{proof}[1][Proof]{\noindent\textbf{#1.} }{\ \rule{0.5em}{0.5em}}
\input{tcilatex}
\begin{document}

\title{Some characterizations of the preimage of $A_{\infty }$ for the
Hardy-Littlewood maximal operator and consequences}
\author{\'{A}lvaro Corval\'{a}n \\
acorvala@ungs.edu.ar}
\maketitle

\begin{abstract}
The purpose of this paper is to give some characterizations of the weight
functions $w$ such that $Mw\in A_{\infty }$. We show that for those weights
to be in $A_{\infty }$ ensures to be in $A_{1}$. We give a criterion in
terms of the local maximal functions $m_{\lambda }$ and we present a pair of
applications, one of them similar to the Coifman-Rochberg characterization
of $A_{1}$ but using functions of the form $\left( f^{\#}\right) ^{\delta }$
and $\left( m_{\lambda }u\right) ^{\delta }$ instead of $\left( Mf\right)
^{\delta }$.
\end{abstract}

INTRODUCTION

In this work we look at some characterizations of the weights $u$ such that $%
Mu\in A_{\infty }$. This question is mentioned as open in [CU-P] and that
paper refers the reader to [CU] for partial results for monotonic functions
in $\mathbb{R}$, and at our knowledge no previous work brings explicitly a
complete result. We will show that if for a weigth $u$ we have that $Mu\in
A_{\infty }$, actually we must have that $Mu\in A_{1}$. From a result due to
Neugebauer it is known that those weights can be characterized for a
pointwise condition for the maximal operator: $\left( M\left( u^{r}\right)
\left( x\right) \right) ^{\frac{1}{r}}\leq CMu\left( x\right) $ for some $%
C>0,r>1$ and $\forall x\in \mathbb{R}^{n}$, so it is immediately satisfied
for a weigth belonging to any reverse H\"{o}lder class -this means that $%
\left( u_{Q}^{r}\right) ^{\frac{1}{r}}\leq C\left( u_{Q}\right) $ for some $%
C>0,r>1$ and any cube $Q$ with sides parallel to the coordinate axes.
Notwithstanding, some weaker conditions, for instance: $u\in weak-A_{\infty
} $, allows to satisfy the condition of Neugebauer. We wil also present
another condition in terms of the size of sub-level sets, by means the use
of some useful pointwise inequalities found by A. Lerner, involving the
sharp maximal operator $u^{\#}$, the local maximal function $m_{\lambda
}\left( u\right) $ and the Hardy-Littlewood maximal operator $Mu$. The
resulting condition is weaker but quite similar to certain characterization
for $A_{\infty }$ weights -in [DMO] it is proven that this characterization,
equivalent to $A_{\infty }$ for standard cubes, is weaker, for general
bases, than most of the usual definitions for $A_{\infty }$ classes-. An
interesting consequence that we can derive from this result is a
characterization of the $A_{1}$ weights similar to the construction of
Coifman and Rochberg in terms of $k\left( x\right) \left( Mf\left( x\right)
\right) ^{\delta }$ -with $k$ and $k^{-1}$ belonging to $L^{\infty }$-, but
involving $u^{\#}$ and $m_{\lambda }\left( u\right) $ instead of $Mf\left(
x\right) $. As another consequence for those weights $u$ such that $Mu\in
A_{\infty }$ and hence $Mu\in A_{1}$ we can improve some known inequalities
for singular integral operators.

The weights belonging to $A_{\infty }$ can be described by several
conditions. In the reference [DMO] many of them are enumerated; all of them
are equivalent for the usual Muckenhoupt weights for the maximal operator
asociated with the bases of cubes whose sides are parallel to the cordinate
axes (or asociated with balls), but they that can provide different types of
weights for other bases. Here we deal with the usual bases of cubes (with
sides parallel to the coordinate axes) and the corresponding Muckenhoupt
weights. But we might translate some of the results for other bases for
which the following condition describe $A_{\infty }$ as the union of $A_{p}$
classes and for which it holds those properties that we use relating the
corresponding weights and the $A_{p}$ constants.

Summarizing the main results are:

\begin{proposition}
If $u$ is any weight, $Mu\in A_{\infty }\iff Mu\in A_{1}$
\end{proposition}

\begin{criterion}
Let $u$ a weight function in $\mathbb{R}^{n}$, $Mu\in A_{\infty }$ if and
only if there exists $s>1$ and $C_{0}>0$ such that $\left( Mu^{s}\right) ^{%
\frac{1}{s}}\left( x\right) \leq C_{0}.Mu\left( x\right) $.
\end{criterion}

\begin{criterion}
Let%
\'{}%
s $u$ a weight function, $Mu\in A_{\infty }$ if and only if for any $\lambda
\in \left( 0,1\right) $ it holds that $m_{\lambda }\left( Mu\right)
\thickapprox M\left( Mu\right) $
\end{criterion}

\begin{theorem}
Let $u$ a weight function. Then $Mu\in A_{\infty }$ if and only if (\ref%
{MUAINF}) holds, that is: 
\begin{equation*}
Mu\in A_{\infty }\iff \exists \alpha >0,\beta \in \left( 0,1\right)
:\left\vert \{y\in Q_{x}:Mu\left( y\right) \leq \alpha .\left( Mu\right)
_{Q}\}\right\vert \leq \beta .\left\vert Q_{x}\right\vert
\end{equation*}%
for almost every $x\in \mathbb{R}^{n}$ for some cube $Q_{x}\ni x$,$\ $and
for every cube $Q$ to which x belongs.
\end{theorem}

\begin{theorem}
(1) If $0<\delta <1$, $f\in L_{loc}^{1}\left( \mathbb{R}^{n}\right) $ and $%
u\in A_{1}$ and $c,d$ non-negative constants then $\left( c\cdot
f^{\#}\left( x\right) +d\cdot m_{\lambda }u\left( x\right) \right) ^{\delta
}\in A_{1}$.

(2) Conversely, if $w\in A_{1}$ then there are $f\in L_{loc}^{1}\left( 
\mathbb{R}^{n}\right) $, $u\in A_{1},$ non-negative constants $c\ $and $d$,
and $k\left( x\right) $with $k,k^{-1}\in L^{\infty }$ such that $w\left(
x\right) =k\left( x\right) \left( c.f^{\#}\left( x\right) ^{\delta
}+d.m_{\lambda }u\left( x\right) ^{\delta }\right) $.
\end{theorem}

\bigskip

PRELIMINARIES

Here $M$ is the (non-centered) Hardy-Littlewood maximal operator for the
bases of cubes with sides parallel to the co-ordinate axes; so if $f\in
L_{loc}^{1}\left( \mathbb{R}^{n}\right) $ we have:%
\begin{equation*}
Mf\left( x\right) =\sup_{Q\ni x}\frac{1}{\left\vert Q\right\vert }%
\int_{Q}f\left( z\right) dz
\end{equation*}

A weight $w$ is a non-negative locally integrable function in $\mathbb{R}%
^{n} $. A weight $w\in A_{p}$ class for $1<p<\infty $ if and only if 
\begin{equation*}
\lbrack A_{p}]:=\sup_{Q}\left( \frac{1}{\left\vert Q\right\vert }%
\int_{Q}w\right) \left( \frac{1}{\left\vert Q\right\vert }\int_{Q}w^{-\frac{1%
}{p-1}}\right) ^{p-1}<+\infty
\end{equation*}

A weight $w\in A_{1}$ if and only if 
\begin{equation*}
Mw\left( x\right) \leq Cw\left( x\right) \text{ }a.e.x\in \mathbb{R}^{n}
\end{equation*}%
and $[A_{1}]$ is the minimal constant $C$ such that this inequality occurs.

We will note $f\left( Q\right) =\int_{Q}f\left( x\right) dx$ and $f_{Q}=%
\frac{f\left( Q\right) }{\left\vert Q\right\vert }$.

We also recall the statement of an useful result due to Coifman, R. and
Rochberg, R. in characterizing $A_{1}$ weights:

\begin{theorem}
(1) Let $f\in L_{loc}^{1}\left( \mathbb{R}^{n}\right) $ be such that $%
Mf\left( x\right) <\infty $ a.e. and $0\leq \delta <1$, then $w\left(
x\right) =Mf\left( x\right) ^{\delta }\ $is in $A_{1}$. Also the $A_{1}$
constant depends only on $\delta $.

(2) Conversely, if $w\in A_{1}$ then there are $f\in L_{loc}^{1}\left( 
\mathbb{R}^{n}\right) $ and $k\left( x\right) $ with $k$ and $k^{-1}$ both
belonging to $L^{\infty }$ such that $w\left( x\right) =k\left( x\right)
Mf\left( x\right) ^{\delta }$.

The proof can be found in [D] (or see [C-R] for the original work), using a
suitable decomposition of f and Kolmogorov%
\'{}%
s inequality for proving (1). The point (2) is quite elementary.
\end{theorem}

We collect some known properties that we will use. The first of which can be
easily obtained using the definition of $A_{p}$ classes and the definition
of $[A_{p}]$ constants, and H\"{o}lder%
\'{}%
s inequality (see [D], for instance):

A) $A_{p}\subset A_{q}$ if $p<q$ and $[w]_{A_{q}}\leq \lbrack w]_{A_{p}}$.

B) $w\in A_{p}$ if and only if $w^{\frac{1}{1-p}}\in A_{\frac{1}{1-p}}$

C) If $w_{0},w_{1}\in A_{1}$ then $w_{0}w_{1}^{1-p}\in A_{p}$

Another property that we will need is the reciprocal of property C). That
property (P. Jones' Factorization Theorem) it%
\'{}%
s very much deeper than the previous (see for instance [S]).

D) If $w\in A_{p}$ there exists $w_{0},w_{1}\in A_{1}$ such that $%
w=w_{0}w_{1}^{1-p}$

Finally, one last property that we will need is:

E) If $w\in A_{p}$ there is $\alpha >1$ such that $w^{\alpha }\in A_{p}$

This latter property is usually proved by means the use of reverse H\"{o}%
lder inequalities that $A_{p}$ weights satisfy (see [D],[G] or [G-R]), but
it can be obtained easily from the Coifman-Rochberg construction: if $w\in
A_{1}$ by (2) is $w\left( x\right) ^{\alpha }=k\left( x\right) ^{\alpha
}Mf\left( x\right) ^{\delta \alpha }$ and taking $1<\alpha <\frac{1}{\delta }
$ we have from (1) that $Mf\left( x\right) ^{\delta \alpha }\in A_{1}$ and
then 
\begin{equation*}
Mw\left( x\right) ^{\alpha }\leq M\left( \left\Vert k\right\Vert _{\infty
}^{\alpha }Mf\left( x\right) ^{\delta \alpha }\right) \leq \lbrack \left(
Mf\right) ^{\delta }]_{A_{1}}\left\Vert k\right\Vert _{\infty }^{\alpha
}\left( Mf\left( x\right) ^{\delta \alpha }\right) \leq
\end{equation*}%
\begin{equation*}
\lbrack \left( Mf\right) ^{\delta }]_{A_{1}}\left\Vert k\right\Vert _{\infty
}^{\alpha }\left\Vert k^{-1}\right\Vert _{\infty }^{\alpha }k\left( x\right)
^{\alpha }Mf\left( x\right) ^{\delta \alpha }=[\left( Mf\right) ^{\delta
}]_{A_{1}}\left\Vert k\right\Vert _{\infty }\left\Vert k^{-1}\right\Vert
_{\infty }^{\alpha }w\left( x\right) ^{\alpha }
\end{equation*}%
So $w\left( x\right) ^{\alpha }\in A_{p}$ with $[w]_{A_{p}}\leq \lbrack
\left( Mf\right) ^{\delta }]_{A_{1}}\left\Vert k\right\Vert _{\infty
}\left\Vert k^{-1}\right\Vert _{\infty }^{\alpha }$. On the other hand for $%
p>1$ and $w\in A_{p}$ by property D) we have $w=w_{0}w_{1}^{1-p}$ with $%
w_{0},w_{1}\in A_{1}$ and for $j=0,1$ we write $w_{j}\left( x\right)
=k_{j}\left( x\right) Mf_{j}\left( x\right) ^{\delta _{j}}$ and for $%
1<\alpha <\min \left\{ \frac{1}{\delta _{j}}\right\} $ we have that $%
w_{0}^{\alpha },w_{1}^{\alpha }\in A_{1}$ and using C) we have that $%
w^{\alpha }=w_{0}^{\alpha }\left( w_{1}^{\alpha }\right) ^{1-p}\in A_{p}$.

By property A, the $A_{p}$ classes are nested, so it is well defined the
class $A_{\infty }=\dbigcup\limits_{p<\infty }A_{p}$.

A characterization of a weight $w$ for belonging to $A_{\infty }$ is the
following:%
\begin{equation}
w\in A_{\infty }\iff \exists \alpha ,\beta \in \left( 0,1\right) :\left\vert
\{y\in Q:w\left( y\right) \leq \alpha .w_{Q}\}\right\vert \leq \beta
.\left\vert Q\right\vert  \label{AINF}
\end{equation}%
for$\ $every cube $Q$ (see for instance [DMO] for this and other
characterizations for general bases).\ 

We will prove that for a weight $u$ there is a necessary and sufficient
condition for $Mu$ to belong to $A_{\infty }$ with a statement quite similar
to (\ref{AINF}). 
\begin{equation}
Mu\in A_{\infty }\iff \exists \alpha >0,\beta \in \left( 0,1\right)
:\left\vert \{y\in Q_{x}:Mu\left( y\right) \leq \alpha .\left( Mu\right)
_{Q}\}\right\vert \leq \beta .\left\vert Q_{x}\right\vert  \label{MUAINF}
\end{equation}%
for $x\in \mathbb{R}^{n}$ a.e. and for some cube $Q_{x}\ni x$,$\ $and for
every cube $Q$ to which x belongs.

\bigskip

SOME\ RESULTS

The first step is the following proposition that shows that if $Mu\in
A_{\infty }$ indeed $Mu\in A_{1}$, and then because $A_{1}\subset A_{\infty
} $ we have that $Mu\in A_{\infty }\iff Mu\in A_{1}$. So, what we have to do
is to characterize the weights $u$ such that $Mu\in A_{1}$.

\begin{remark}
Of course $A_{1}\subsetneqq A_{\infty }$, so there are weights $w$ such that 
$w\in A_{\infty }$ and $w\notin A_{1}$. The lemma tells us that being in $%
A_{\infty }$ is the same as being in $A_{1}$ for those weights $w$ such that 
$w=Mu$ for some weight $u$.
\end{remark}

\begin{proposition}
If $u$ is any weight, $Mu\in A_{\infty }\iff Mu\in A_{1}$
\end{proposition}

\begin{proof}
The implication $Mu\in A_{1}\implies Mu\in A_{\infty }$ is trivial because $%
A_{1}\subset A_{\infty }$.

It remains to show that if $Mu\in A_{\infty }\implies Mu\in A_{1}$.

If $Mu\in A_{\infty }=\dbigcup\limits_{p<\infty }A_{p}$, we have that $Mu\in
A_{p}$ for some $p\geq 1$. If $p=1$ there is nothing to prove. Let $p>1$.
Because the result of Coifman and Rochberg we have that $\left( Mu\right)
^{\delta }\in A_{1}$ for any $\delta $ with $0\leq \delta <1$ and any $u$
locally integrable but generally does not occur that $Mu\in A_{1}$, actually
we are in the process of proving that if we additionally have that $Mu\in
A_{p}$, in fact $Mu\in A_{1}$.

We need the following result (see, for instance, [Rudin, ej 5 d) Chap 3]):
For a measure space $\left( \Omega ,\mu \right) $ with measure $\mu \left(
\Omega \right) =1$ and $\left( \int_{\Omega }\left\vert f\right\vert
^{r}d\mu \right) ^{\frac{1}{r}}<\infty $ for some $r>0$, we have that 
\begin{equation*}
\lim\limits_{r\rightarrow 0^{+}}\left( \int_{\Omega }\left\vert f\right\vert
^{r}d\mu \right) ^{\frac{1}{r}}=\exp \left( \int_{\Omega }\log \left(
\left\vert f\right\vert \right) d\mu \right)
\end{equation*}%
.

Let's observe that using that $\mu \left( \Omega \right) =1$ and H\"{o}lder
Inequality we obtain $\left( \int_{\Omega }\left\vert f\right\vert
^{r_{1}}d\mu \right) ^{\frac{1}{r_{1}}}\geq \left( \int_{\Omega }\left\vert
f\right\vert ^{r_{2}}d\mu \right) ^{\frac{1}{r_{2}}}$ if $r_{1}\geq r_{2}$.
So for $r>0$ we have that 
\begin{equation*}
\left( \int_{\Omega }\left\vert f\right\vert ^{r}d\mu \right) ^{\frac{1}{r}%
}\geq \exp \left( \int_{\Omega }\log \left( \left\vert f\right\vert \right)
d\mu \right) =\lim\limits_{r\rightarrow 0^{+}}\left( \int_{\Omega
}\left\vert f\right\vert ^{r}d\mu \right) ^{\frac{1}{r}}
\end{equation*}%
.

Now for every $q>p$, using that 
\begin{equation*}
\sup\limits_{Q}\frac{Mu(Q)}{\left\vert Q\right\vert }\left( \frac{1}{%
\left\vert Q\right\vert }\int_{Q}Mu\left( x\right) ^{-\frac{1}{q-1}%
}dx\right) ^{q-1}=[Mu]_{A_{q}}\leq \lbrack Mu]_{A_{p}}
\end{equation*}%
(property A), we obtain that for any cube $Q:$ 
\begin{equation*}
\frac{Mu(Q)}{\left\vert Q\right\vert }\left( \frac{1}{\left\vert
Q\right\vert }\int_{Q}Mu\left( x\right) ^{-\frac{1}{q-1}}dx\right)
^{q-1}\leq \lbrack Mu]_{A_{p}}<\infty
\end{equation*}%
.

If $q$ tends to infinity then $\frac{1}{q-1}$ tends to $0^{+}$,$\ $so taking 
$r=\frac{1}{q-1}$ and applying the result from above for $f=w^{-1}$, $\Omega
=Q$ and $d\mu =\frac{dx}{\left\vert Q\right\vert }$, we have 
\begin{equation*}
\lim\limits_{q\rightarrow +\infty }\left( \frac{1}{\left\vert Q\right\vert }%
\int_{Q}Mu\left( x\right) ^{-\frac{1}{q-1}}dx\right) ^{q-1}=\exp \left(
\int_{Q}\log \left( Mu\left( x\right) ^{-1}\right) dx\right)
\end{equation*}%
\begin{equation*}
=\exp \left( \int_{Q}-\log \left( Mu\left( x\right) \right) dx\right) =\frac{%
1}{\exp \left( \int_{Q}\log \left( Mu\left( x\right) \right) dx\right) }
\end{equation*}%
. Taking limit in $\frac{Mu(Q)}{\left\vert Q\right\vert }\left( \frac{1}{%
\left\vert Q\right\vert }\int_{Q}Mu\left( x\right) ^{-\frac{1}{q-1}%
}dx\right) ^{q-1}\leq \lbrack Mu]_{A_{p}}$ we have that 
\begin{equation*}
\frac{Mu(Q)}{\left\vert Q\right\vert }\frac{1}{\exp \left( \int_{Q}\log
\left( Mu\left( x\right) \right) dx\right) }\leq \lbrack Mu]_{A_{p}}
\end{equation*}%
, so 
\begin{equation*}
\frac{Mu(Q)}{\left\vert Q\right\vert }\leq \lbrack Mu]_{A_{p}}\cdot \exp
\left( \int_{Q}\log \left( Mu\left( x\right) \right) dx\right)
\end{equation*}

Additionally, the observation from above applied for $f=Mu$ gives us that
for any $r>0$ it holds that 
\begin{equation*}
\left( \frac{1}{\left\vert Q\right\vert }\int_{Q}\left( Mu\right)
^{r}dx\right) ^{\frac{1}{r}}\geq \exp \left( \int_{Q}\log \left( Mu\left(
x\right) \right) dx\right)
\end{equation*}%
. Thus 
\begin{equation*}
\frac{Mu(Q)}{\left\vert Q\right\vert }\leq \lbrack Mu]_{A_{p}}\cdot \exp
\left( \int_{Q}\log \left( Mu\left( x\right) \right) dx\right) \leq \lbrack
Mu]_{A_{p}}\left( \frac{1}{\left\vert Q\right\vert }\int_{Q}\left\vert
Mu\right\vert ^{r}dx\right) ^{\frac{1}{r}}\,
\end{equation*}%
, and then 
\begin{equation*}
\frac{Mu(Q)}{\left\vert Q\right\vert }\leq \lbrack Mu]_{A_{p}}\left( \frac{1%
}{\left\vert Q\right\vert }\int_{Q}\left\vert Mu\right\vert ^{r}dx\right) ^{%
\frac{1}{r}}
\end{equation*}%
.

Taking $r=\delta $ with $0\leq \delta <1$ and using that for such $\delta $
it holds that $\left( Mu\right) ^{r}=\left( Mu\right) ^{\delta }\in A_{1}$
and then 
\begin{equation*}
\frac{1}{\left\vert Q\right\vert }\int_{Q}\left\vert Mu\right\vert
^{r}dx\leq \lbrack \left( Mu\right) ^{r}]_{A_{1}}\cdot \left( Mu\left(
x\right) \right) ^{r}
\end{equation*}%
a.e for every $x\in Q$.

So we have a.e for $x\in Q$ 
\begin{equation*}
\frac{Mu(Q)}{\left\vert Q\right\vert }\leq \lbrack Mu]_{A_{p}}\left( \frac{1%
}{\left\vert Q\right\vert }\int_{Q}\left\vert Mu\right\vert ^{r}dx\right) ^{%
\frac{1}{r}}
\end{equation*}%
\begin{equation*}
\leq \lbrack \left( Mu\right) ^{r}]_{A_{1}}\cdot \left( \lbrack \left(
Mu\right) ^{r}]_{A_{1}}\cdot \left( Mu\left( x\right) \right) ^{r}\right) ^{%
\frac{1}{r}}
\end{equation*}%
\begin{equation*}
=[\left( Mu\right) ^{r}]_{A_{1}}\cdot \left( \lbrack \left( Mu\right)
^{r}]_{A_{1}}\right) ^{\frac{1}{r}}\cdot \left( Mu\left( x\right) \right)
\end{equation*}%
.

Taking $C=[\left( Mu\right) ^{r}]_{A_{1}}\cdot \left( \lbrack \left(
Mu\right) ^{r}]_{A_{1}}\right) ^{\frac{1}{r}}$ independent of $Q$, for every 
$Q$ we obtain that 
\begin{equation*}
\frac{Mu(Q)}{\left\vert Q\right\vert }\leq C\cdot Mu\left( x\right)
\end{equation*}%
a.e for $x\in Q$.

Then almost everywhere for $x\in \mathbb{R}^{n}$ we have that 
\begin{equation*}
M\left( Mu\right) \left( x\right) =\sup\limits_{Q\ni x}\frac{Mu(Q)}{%
\left\vert Q\right\vert }\leq C\cdot Mu\left( x\right)
\end{equation*}%
, that is 
\begin{equation*}
M\left( Mu\right) \left( x\right) \leq C\cdot Mu\left( x\right)
\end{equation*}%
and then we obtain that $Mu\in A_{1}$.
\end{proof}

The previous proposition together with a lemma due to Neugebauer (published
in [CU]) enables us to give a characterization of all the weights $u$ such
that $Mu\in A_{\infty }$. Until a few years ago this was an open problem
with interesting consequences for improving some two-weight inequalities for
several operators, including maximal, vector-valued an Calderon-Zygmund ones
(see [CU-P]).

For completitude we transcribe below the lemma of Neugebauer and its easy
proof, in [CU] the lemma is considered in $\mathbb{R}$ but it works mutatis
mutandi for $\mathbb{R}^{n}$.

\begin{lemma}
(Neugebauer) For a weight $u$ it holds that $Mu\in A_{1}$ if and only if
there exists $s>1$ and $C_{0}>0$ such that $\left( Mu^{s}\right) ^{\frac{1}{s%
}}\left( x\right) \leq C_{0}.Mu\left( x\right) $
\end{lemma}

\begin{proof}
If such $s>1$ exists then $\frac{1}{s}<1$ and the Coifman-Rochberg
characterization of $A_{1}$ weights tells us that $\left( Mu^{s}\right) ^{%
\frac{1}{s}}$ is in $A_{1}$, so $M\left( \left( Mu^{s}\right) ^{\frac{1}{s}%
}\right) \leq C_{1}\left( Mu^{s}\right) ^{\frac{1}{s}}$, and using the
hypothesis and the fact that by H\"{o}lder: $Mu\leq \left( Mu^{s}\right) ^{%
\frac{1}{s}}$, we obtain $M\left( Mu\right) \leq M\left( \left(
Mu^{s}\right) ^{\frac{1}{s}}\right) \leq C_{1}.\left( Mu^{s}\right) ^{\frac{1%
}{s}}\leq C_{1}.C.Mu$, and then $M(Mu)\leq C.Mu$, that is $Mu\in A_{1}$.

Reciprocally if $Mu\in A_{1}$ then $Mu$ satisfies a reverse H\"{o}lder
inequality (RHI), that means that for some $s>1$ and $C>0$ it holds for any
cube $Q$%
\begin{equation*}
\left( \frac{1}{\left\vert Q\right\vert }\int_{Q}Mu^{s}\right) ^{\frac{1}{s}%
}\leq C.\frac{1}{\left\vert Q\right\vert }\int_{Q}Mu
\end{equation*}%
and taking suprema over the cubes we have: 
\begin{equation*}
\left( Mu^{s}\right) ^{\frac{1}{s}}\leq C.Mu
\end{equation*}
\end{proof}

As we have already mention the lemma and the proposition above, which says
that $Mu\in A_{\infty }$ if and only if it actually belongs to $A_{1}$,
provide us with the following characterization of the weights $u$ such that $%
Mu\in A_{\infty }$:

\begin{criterion}
Let $u$ a weight function in $\mathbb{R}^{n}$, $Mu\in A_{\infty }$ if and
only if there exists $s>1$ and $C_{0}>0$ such that $\left( Mu^{s}\right) ^{%
\frac{1}{s}}\left( x\right) \leq C_{0}.Mu\left( x\right) $.
\end{criterion}

Let%
\'{}%
s observe that with have got a bound for the constant $[Mu]_{A_{1}}$, that
is 
\begin{equation*}
\lbrack Mu]_{A_{1}}\leq \lbrack \left( Mu\right) ^{r}]_{A_{1}}\cdot \left(
\lbrack \left( Mu\right) ^{r}]_{A_{1}}\right) ^{\frac{1}{r}}
\end{equation*}%
.

Because the previous proposition the weights $u$ with $Mu$ in $A_{\infty }$
are those for which there are some $C>0$ such that 
\begin{equation*}
M(Mu)\left( x\right) \leq C\cdot Mu\left( x\right) \text{ a.e.}
\end{equation*}%
.

\bigskip

SOME\ FURTHER\ DEFINITIONS\ AND\ PROPERTIES

Now we will use some pointwise inequalities for certain maximal operators to
weaken the above condition. We need a couple of definitions:

\begin{definition}
If $f\in L_{loc}^{1}\left( \mathbb{R}^{n}\right) $ the sharp maximal
function of Fefferman-Stein $f^{\#}$ is defined by%
\begin{equation*}
f^{\#}\left( x\right) =\sup_{Q\ni x}\frac{1}{\left\vert Q\right\vert }%
\int_{Q}\left\vert f\left( x\right) -f_{Q}\right\vert dx
\end{equation*}
\end{definition}

\begin{definition}
$BMO\left( \mathbb{R}^{n}\right) =\{f\in L_{loc}^{1}\left( \mathbb{R}%
^{n}\right) :f^{\#}\in L^{\infty }\left( \mathbb{R}^{n}\right) \}$ is the
space of functions with bounded mean oscillation, and $\left\Vert
f\right\Vert _{BMO}=\left\Vert f^{\#}\right\Vert _{\infty }$.
\end{definition}

\begin{remark}
$\left\Vert {}\right\Vert _{BMO}$ is a seminorm for $BMO\left( \mathbb{R}%
^{n}\right) $ since $\left\Vert f^{\#}\right\Vert _{\infty }=0$ if and only
if $f$ is constant (a.e.). It is usual to identify $BMO$ with its quotient
with the class of almost everywhere constant functions and then $\left\Vert
{}\right\Vert _{BMO}$ becames a norm.
\end{remark}

\begin{notation}
For a measurable function $f:\mathbb{R}^{n}\longrightarrow \mathbb{R}$, the
non-increasing rearrangement of $f$ is $f^{\ast }$. That is, for $t\geq 0$%
\begin{equation*}
f^{\ast }\left( t\right) =\inf \{\alpha >0:\left\vert \left\{ x\in \mathbb{R}%
^{n}:\left\vert f\left( x\right) \right\vert >\alpha \right\} \right\vert
\leq t\}
\end{equation*}%
. We use the convention that $\inf \emptyset =\infty $.
\end{notation}

\begin{remark}
An equivalent way to define $f^{\ast }\left( t\right) $ is 
\begin{equation*}
f^{\ast }\left( t\right) =\sup_{\left\vert E\right\vert =t}\inf_{x\in
E}\left\vert f\left( x\right) \right\vert
\end{equation*}%
where $E$ are measurable sets.
\end{remark}

\begin{remark}
Non-increasing rearrangements of functions from measure spaces $\left( X,\mu
\right) $ can be defined in the same way replacing $\mathbb{R}^{n}$ by $X$
and the Lebesgue measure $\left\vert {}\right\vert $ by $\mu $. Much more
details and results can be found in [BS].
\end{remark}

\begin{definition}
If $f$ is a measurable function and $\lambda \in \left( 0,1\right) $ the
local maximal functions $m_{\lambda }\left( f\right) $ are defined by 
\begin{equation*}
m_{\lambda }f\left( x\right) =\sup_{Q\ni x}\left( f\chi _{Q}\right) ^{\ast
}\left( \lambda \left\vert Q\right\vert \right)
\end{equation*}
\end{definition}

Let's point out some basic properties of $f^{\ast }$, $m_{\lambda }f\left(
x\right) $, and $f^{\#}$, immediate from their definitions:

(i) $f^{\#}\left( x\right) \leq 2Mf\left( x\right) $

(ii) If $c>0$ then $\left( c.f\right) ^{\ast }\left( t\right) =c.\left(
f\right) ^{\ast }\left( t\right) $

(iii) If $f\left( x\right) \geq g\left( x\right) $ $a.e.$ then $f^{\ast
}\left( t\right) \geq g^{\ast }\left( t\right) $ for every $t$.

(iv) Using iii) if $f\left( x\right) \geq g\left( x\right) $ $a.e.$ then $%
m_{\lambda }\left( f\right) \left( x\right) \geq m_{\lambda }\left( g\right)
\left( x\right) $ everywhere.

(v) If $c>0$ using ii) we have $m_{\lambda }\left( c.f\right) \left(
x\right) =c.m_{\lambda }\left( f\right) \left( x\right) $.

We will also need the somewhat less trivial inequalities:

\begin{lemma}
(vi) $m_{\lambda }\left( f\right) \left( x\right) \geq \left\vert f\left(
x\right) \right\vert $ that holds at every Lebesgue point of $f$, so a.e. if 
$f\in L_{loc}^{1}\left( \mathbb{R}^{n}\right) $.
\end{lemma}

\begin{proof}
We will need to remember a definition and a known result of Real Analysis.
The definition is the following: a sequence $\{E_{i}\}_{i\in \mathbb{N}}$ of
Borel sets of $\mathbb{R}^{n}$ is said to shrink to $x$ nicely if there is a
number $\alpha >0$ such that there is a sequence of cubes of $\mathbb{R}^{n}$
centered at $x$ of radii $r_{i}\rightarrow 0$, $\{Q_{\left( x,r_{i}\right)
}\}_{i\in \mathbb{N}}$, such that $E_{i}\subset Q_{\left( x,r_{i}\right) }$
and $\left\vert E_{i}\right\vert \geq \alpha .\left\vert Q_{\left(
x,r_{i}\right) }\right\vert $. The result is: if $x\in \mathbb{R}^{n}$ is a
Lebesgue point of $f\in L_{loc}^{1}\left( \mathbb{R}^{n}\right) $ and $%
\{E_{i}\}_{i\in \mathbb{N}}$ is a sequence of sets that shrinks to $x$
nicely then 
\begin{equation*}
f\left( x\right) =\lim\limits_{i\rightarrow \infty }\frac{1}{\left\vert
E_{i}\right\vert }\int_{E_{i}}f\left( z\right) dz
\end{equation*}%
.(see [Rudin], theorem 7.10 -changing cubes for balls and $f\in
L_{loc}^{1}\left( \mathbb{R}^{n}\right) $ instead of $f\in L^{1}\left( 
\mathbb{R}^{n}\right) $ the proof still works-).

Now for any positive $\tau $ with $\tau <1$, using the definitions of
non-increasing rearrangements and $m_{\lambda }$ we have that 
\begin{equation*}
\forall Q\ni x:\left\vert \{y\in Q:\left\vert f\left( y\right) \right\vert
>\tau .m_{\lambda }f\left( x\right) \}\right\vert \geq \lambda \left\vert
Q\right\vert
\end{equation*}%
. So if we take $r_{i}=\frac{1}{i}\rightarrow 0$ and we name 
\begin{equation*}
\{E_{i}\}_{i\in \mathbb{N}}=\{y\in Q_{\left( x,r_{i}\right) }:\left\vert
f\left( y\right) \right\vert \leq \tau .m_{\lambda }f\left( x\right) \}
\end{equation*}%
then 
\begin{equation*}
E_{i}=Q_{\left( x,r_{i}\right) }\setminus \{y\in Q_{\left( x,r_{i}\right)
}:\left\vert f\left( y\right) \right\vert >\tau .m_{\lambda }f\left(
x\right) \}
\end{equation*}
and we obtain that 
\begin{equation*}
\left\vert E_{i}\right\vert =\left\vert Q_{\left( x,r_{i}\right) }\setminus
\{y\in Q_{\left( x,r_{i}\right) }:\left\vert f\left( y\right) \right\vert
>\tau .m_{\lambda }f\left( x\right) \}\right\vert \geq Q_{\left(
x,r_{i}\right) }-\lambda \left\vert Q_{\left( x,r_{i}\right) }\right\vert
\end{equation*}%
that is%
\begin{equation*}
\left\vert E_{i}\right\vert \geq \left( 1-\lambda \right) .\left\vert
Q_{\left( x,r_{i}\right) }\right\vert
\end{equation*}%
and then $\{E_{i}\}_{i\in \mathbb{N}}$ is a sequence of sets that shrinks to 
$x$ nicely. But now, with these sets $E_{i}$ we can apply the mentioned
result for any Lebesgue point to obtain: 
\begin{equation*}
f\left( x\right) =\lim\limits_{i\rightarrow \infty }\frac{1}{\left\vert
E_{i}\right\vert }\int_{E_{i}}f\left( z\right) dz\leq
\lim\limits_{i\rightarrow \infty }\frac{1}{\left\vert E_{i}\right\vert }%
\int_{E_{i}}\tau .m_{\lambda }f\left( x\right) dz
\end{equation*}%
and using $\left\vert f\left( x\right) \right\vert $ instead of $f\left(
x\right) :$%
\begin{equation*}
\left\vert f\left( x\right) \right\vert \leq \lim\limits_{i\rightarrow
\infty }\frac{\tau .m_{\lambda }f\left( x\right) }{\left\vert
E_{i}\right\vert }\int_{E_{i}}dz=\lim\limits_{i\rightarrow \infty }\frac{%
\tau .m_{\lambda }f\left( x\right) }{\left\vert E_{i}\right\vert }\left\vert
E_{i}\right\vert =\tau .m_{\lambda }f\left( x\right)
\end{equation*}%
Then 
\begin{equation*}
\left\vert f\left( x\right) \right\vert \leq \tau .m_{\lambda }f\left(
x\right)
\end{equation*}%
$\forall \tau <1$, and taking limit for $\tau \rightarrow 1^{-}$ we obtain: 
\begin{equation*}
\left\vert f\left( x\right) \right\vert \leq m_{\lambda }f\left( x\right)
\end{equation*}%
for every Lebesgue point of $f$ and then almost everywhere.
\end{proof}

(vii) For any $\lambda \in \left( 0,1\right) $ there is a constant $%
c_{\lambda ,n}$ (depending only of $\lambda $ and $n$) such that for all $%
u\in L_{loc}^{1}$ and $x\in \mathbb{R}^{n}$ we have ([L]):%
\begin{equation*}
m_{\lambda }\left( Mu\right) \left( x\right) \leq c_{\lambda
,n}.u^{\#}\left( x\right) +Mu\left( x\right)
\end{equation*}

(viii) Observe that using vii) and aplying ii) to $f=Mu$ we obtain $%
m_{\lambda }\left( Mu\right) \left( x\right) \leq c.Mu\left( x\right) $ a.e.
for some $c>0$.

(ix) $m_{\lambda }\left( Mu\right) $ and $Mu$ are pointwise equivalent a.e.
(we will write $m_{\lambda }\left( Mu\right) \thickapprox Mu$ for that
situation) that is that is there are positive constants $A$ and $B$ such
that $m_{\lambda }(Mu)\left( x\right) \leq A.Mu\left( x\right) $ and $%
Mu\left( x\right) \leq B.m_{\lambda }(Mu\left( x\right) )$ a.e., we obtain
this taking $A=c$ in viii), and $B=1$ in vi).

(x) It's immediate from the definition of $M$ that $Mf\left( x\right) \geq
f\left( x\right) $ a.e.

(xi) We will also use a pointwise inequality (see [L2]) that goes in the
opposite direction of vii): for any $u\in L_{loc}^{1}$ and $x\in \mathbb{R}%
^{n}$ we have :

\bigskip

SOME\ MORE\ RESULTS

Another criterion for characterization of the weights $u$ such that $Mu\in
A_{\infty }$ follows from our Proposition 1 and from inequality vii) :

\begin{criterion}
Let%
\'{}%
s $u$ a weight function, $Mu\in A_{\infty }$ if and only if for any $\lambda
\in \left( 0,1\right) $ it holds that $m_{\lambda }\left( Mu\right)
\thickapprox M\left( Mu\right) $.
\end{criterion}

\begin{proof}
By Proposition 1 we have $Mu\in A_{\infty }\iff Mu\in A_{1}$, so $Mu\in
A_{\infty }$ if and only if there is some $C>0:M\left( Mu\right) \left(
x\right) \leq C.Mu\left( x\right) $ a.e. and using that $M\left( f\right)
\left( x\right) \geq f\left( x\right) $ a.e. for $f\in L_{loc}^{1}$ we have
that $M\left( Mu\right) \left( x\right) \geq Mu\left( x\right) $ a.e., and
then ix) gives us that $Mu\in A_{\infty }\iff Mu\in A_{1}\iff Mu\thickapprox
M\left( Mu\right) \iff m_{\lambda }\left( Mu\right) \thickapprox M\left(
Mu\right) $.
\end{proof}

\begin{remark}
We can observe that it is enough that $m_{\lambda }\left( Mu\right)
\thickapprox M\left( Mu\right) $ for some $\lambda \in \left( 0,1\right) $
to obtain that $Mu\in A_{\infty }$ and then $m_{\lambda }\left( Mu\right)
\thickapprox M\left( Mu\right) $ for every $\lambda \in \left( 0,1\right) $.
\end{remark}

\begin{remark}
Because of viii) for any $u$ we always can ensure for a suitable $%
c>0:m_{\lambda }\left( Mu\right) \left( x\right) \leq c.Mu\left( x\right)
\leq c.M\left( Mu\right) \left( x\right) $, that is $m_{\lambda }\left(
Mu\right) \left( x\right) \leq c.M\left( Mu\right) \left( x\right) $ a.e.;
thus, by the criterion above, a condition necessary and sufficient, on $u$,
for $Mu$ to belong to $A_{\infty }$ is the existence of a constant $C>0$
such that $M\left( Mu\right) \left( x\right) \leq C.m_{\lambda }\left(
Mu\right) \left( x\right) $ a.e.
\end{remark}

\ As we mentioned in the introduction now we want to prove that (\ref{MUAINF}%
) is a necessary and sufficient condition on a weight $u$ for $Mu$ to be in $%
A_{\infty }$.

A condition like (\ref{MUAINF}) but applied for an arbitrary weight $w$
instead of $Mu$ is weaker than (\ref{AINF}), that is, if $w\in A_{\infty }$
then $w$ satifies the following:

\begin{condition}[LocalAINF]
$\exists \alpha _{1}>0,\beta _{1}\in \left( 0,1\right) $ such that for
almost every $x\in \mathbb{R}^{n}$ exists a cube $Q_{x}\ni x$ that $\forall
Q\ni x$ verifies that:$\left\vert \{y\in Q_{x}:w\left( y\right) \leq \alpha
_{1}.w_{Q}\}\right\vert \leq \beta _{1}.\left\vert Q_{x}\right\vert $
\end{condition}

To see this implication let%
\'{}%
s remember that $w\in A_{\infty }$ if and only if $w$ satisfies:

\begin{condition}[CAINF]
$\exists \alpha ,\beta \in \left( 0,1\right) :\forall Q$ cube we have $%
\left\vert \{y\in Q:w\left( y\right) \leq \alpha .w_{Q}\}\right\vert \leq
\beta .\left\vert Q\right\vert $
\end{condition}

Now, if $w\in A_{\infty }$ we fix some $k\in \left( 0,1\right) $, for
instance $k=\frac{1}{2}$, and for any $x$ we take a cube $Q_{x}\ni x$ such
that $w_{Q_{x}}=\frac{w\left( Q_{x}\right) }{\left\vert Q_{x}\right\vert }%
\geq k.Mw\left( x\right) $. So let $\alpha _{1}=\alpha .k$ and for any $Q\ni
x$ we have that 
\begin{equation*}
\{y\in Q_{x}:w\left( y\right) \leq \alpha _{1}.w_{Q}\}\subset \{y\in
Q_{x}:w\left( y\right) \leq \alpha _{1}.Mw\left( x\right) \}
\end{equation*}%
\begin{equation*}
\subset \{y\in Q_{x}:w\left( y\right) \leq \frac{\alpha _{1}}{k}\frac{%
w\left( Q_{x}\right) }{\left\vert Q_{x}\right\vert }\}
\end{equation*}%
then applying the previous condition to $Q_{x}$ 
\begin{equation*}
\left\vert \{y\in Q_{x}:w\left( y\right) \leq \alpha _{1}.w_{Q}\}\right\vert
\leq \left\vert \{y\in Q_{x}:w\left( y\right) \leq \frac{\alpha _{1}}{k}%
.w_{Q_{x}}\}\right\vert
\end{equation*}%
\begin{equation*}
=\left\vert \{y\in Q_{x}:w\left( y\right) \leq \alpha
.w_{Q_{x}}\}\right\vert \leq \beta .\left\vert Q_{x}\right\vert
\end{equation*}%
so the condition (LocalAINF) is fulfilled with $\alpha _{1}=\alpha .k$, $%
\beta _{1}=\beta $ and the $Q_{x}$ selected for which $\frac{w\left(
Q_{x}\right) }{\left\vert x\right\vert }\geq k.Mw\left( x\right) $.

Then we have that it also holds:

Although the condition (LocalAINF) is weaker than $A_{\infty }$ for a
general weight when it is applied to a weight that is the maximal function
of another weight, that is if $w=Mu$ then the condition (LocalAINF) implies $%
A_{\infty }$, so they are equivalent conditions for $Mu$ weights.

\begin{theorem}
Let $u$ a weight function. Then $Mu\in A_{\infty }$ if and only if (\ref%
{MUAINF}) holds, that is: 
\begin{equation*}
Mu\in A_{\infty }\iff \exists \alpha >0,\beta \in \left( 0,1\right)
:\left\vert \{y\in Q_{x}:Mu\left( y\right) \leq \alpha .\left( Mu\right)
_{Q}\}\right\vert \leq \beta .\left\vert Q_{x}\right\vert
\end{equation*}%
for almost every $x\in \mathbb{R}^{n}$ for some cube $Q_{x}\ni x$,$\ $and
for every cube $Q$ to which x belongs.
\end{theorem}

\begin{proof}
Because the previous remark $Mu\in A_{\infty }$ if and only if there exists
a positive constant $B$ and $\lambda \in \left( 0,1\right) $ $:$%
\begin{equation}
M(Mu)\left( x\right) \leq B.m_{\lambda }(Mu\left( x\right) )  \label{MUAINFB}
\end{equation}%
a.e. So to guarantee $Mu\in A_{\infty }$ is equivalent to have. 
\begin{equation}
\alpha .M(Mu)\left( x\right) \leq m_{\lambda }(Mu\left( x\right) )
\label{MUAINFC}
\end{equation}%
for some $\alpha >0$ and almost every $x\in \mathbb{R}^{n}$. Now using the
definition of $m_{\lambda }$ we have that (\ref{MUAINFC}) is equivalent to
say that for almost every $x\in \mathbb{R}^{n}$ 
\begin{equation*}
\exists Q_{x}\ni x:\left( Mu.\chi _{Q_{x}}\right) ^{\ast }\left( \lambda
.\left\vert Q_{x}\right\vert \right) \geq \alpha .\left( Mu\right) _{Q}
\end{equation*}%
for every cube $Q\ni x$. Now by the definition of non-increasing
rearrangements this means that for a.e. $x\in \mathbb{R}^{n}$ 
\begin{equation*}
\exists Q_{x}\ni x:\left\vert \{y\in Q_{x}:Mu\left( y\right) >\alpha .\left(
Mu\right) _{Q}\}\right\vert >\lambda .\left\vert Q_{x}\right\vert
\end{equation*}%
for every cube $Q\ni x$, or, taking complements respect $Q_{x}$ and naming $%
\beta =\left( 1-\lambda \right) \in \left( 0,1\right) $, we have that (\ref%
{MUAINFB}) and therefore $Mu\in A_{\infty }$ is equivalent to the existence
of $\alpha >0,\beta \in \left( 0,1\right) $ such that for almost every $x\in 
\mathbb{R}^{n}$ there is some $Q_{x}\ni x:$%
\begin{equation*}
\exists Q_{x}\ni x:\left\vert \{y\in Q_{x}:Mu\left( y\right) \leq \alpha
.\left( Mu\right) _{Q}\}\right\vert \leq \beta .\left\vert Q_{x}\right\vert
\end{equation*}%
for every cube $Q\ni x$.
\end{proof}

\begin{example}
It%
\'{}%
s easy to see that a class of weights functions $u$ such that $Mu\in
A_{\infty }$ is the class $A_{\infty }$ itself, that is $M\left( A_{\infty
}\right) \subset A_{\infty },$ and by our first proposition in fact $M\left(
A_{\infty }\right) \subset A_{1}$. Indeed we can provide an elementary proof
of this using the previous theorem and the characterization (\ref{AINF}) of $%
A_{\infty }$ weights: We fix some $k\in \left( 0,1\right) $, and for any $x$
we take a cube $Q_{x}$ such that $\frac{Mu\left( Q_{x}\right) }{\left\vert
Q_{x}\right\vert }\geq k.M\left( Mu\right) \left( x\right) $; because (\ref%
{AINF}) and the fact that $u\in A_{\infty }$ we have $\alpha _{1},\beta
_{1}\ $such that for any cube $\widetilde{Q}$ it holds: $\left\vert \{y\in 
\widetilde{Q}:u\left( y\right) \leq \alpha _{1}.u_{\widetilde{Q}%
}\}\right\vert \leq \beta _{1}.\left\vert \widetilde{Q}\right\vert $. Then
for $\widetilde{Q}=Q_{x}$, $\alpha =\frac{\alpha _{1}}{k}$, $\beta =\beta
_{1}$ and for any $Q\ni x$, and using the trivial inclusions due to the
inequalities $\frac{Mu\left( Q_{x}\right) }{\left\vert Q_{x}\right\vert }%
\geq k.M\left( Mu\right) \left( x\right) $; $MMu\left( z\right) \geq
Mu\left( z\right) $ a.e. and $Mu\left( z\right) \geq u\left( z\right) $ a.e.
we get:%
\begin{equation*}
\left\vert \{y\in Q_{x}:Mu\left( y\right) \leq \alpha .\frac{Mu\left(
Q\right) }{\left\vert Q\right\vert }\}\right\vert \leq \left\vert \{y\in
Q_{x}:Mu\left( y\right) \leq \alpha .\frac{MMu\left( Q\right) }{\left\vert
Q\right\vert }\}\right\vert \leq
\end{equation*}%
\begin{equation*}
\leq \left\vert \{y\in Q_{x}:Mu\left( y\right) \leq \alpha .M\left(
Mu\right) \left( x\right) \}\right\vert \leq \left\vert \{y\in Q_{x}:u\left(
y\right) \leq \alpha .M\left( Mu\right) \left( x\right) \}\right\vert
\end{equation*}%
\begin{equation*}
\leq \left\vert \{y\in Q_{x}:u\left( y\right) \leq \frac{\alpha }{k}.\frac{%
Mu\left( Q_{x}\right) }{\left\vert Q_{x}\right\vert }\}\right\vert \leq
\beta .\left\vert Q_{x}\right\vert
\end{equation*}%
that is we have%
\begin{equation*}
\left\vert \{y\in Q_{x}:Mu\left( y\right) \leq \frac{\alpha }{k}.\frac{%
Mu\left( Q\right) }{\left\vert Q\right\vert }\}\right\vert \leq \beta
.\left\vert Q_{x}\right\vert
\end{equation*}
\end{example}

\begin{example}
Actually for those functions there are shorter way to prove that $Mu\in
A_{1}:$ Because the H\"{o}lder's inequality we have that for all $r>1:$%
\begin{equation*}
\frac{1}{\left\vert Q\right\vert }\int_{Q}u\left( x\right) \leq \left( \frac{%
1}{\left\vert Q\right\vert }\int_{Q}u^{r}\left( x\right) \right) ^{\frac{1}{r%
}}
\end{equation*}%
, and taking suprema 
\begin{equation*}
Mu\left( x\right) \leq \left( M\left( u^{r}\right) \left( x\right) \right) ^{%
\frac{1}{r}}
\end{equation*}%
. Now for the Coifman-Rochberg characterization of $A_{1}$ weights for any
locally integrable function $g$ and $\delta \in \lbrack 0,1)$ we have that $%
Mg\left( x\right) ^{\delta }\in A_{1}$ and then $\left( M\left( u^{r}\right)
\left( x\right) \right) ^{\frac{1}{r}}\in A_{1}$, therefore for some
constant $C>1:$ 
\begin{equation*}
MMu\left( x\right) \leq M\left( \left( M\left( u^{r}\right) \left( x\right)
\right) ^{\frac{1}{r}}\right) \leq C.\left( M\left( u^{r}\right) \left(
x\right) \right) ^{\frac{1}{r}}
\end{equation*}%
a.e. But if $u\in A_{\infty }$ then $u\in A_{p}$ for some $p\geq 1$, and
then it satisfy a reverse H\"{o}lder inequality (see [D]) for some $r>1$,
that is 
\begin{equation*}
\left( \frac{1}{\left\vert Q\right\vert }\int_{Q}u^{r}\left( x\right)
\right) ^{\frac{1}{r}}\leq C.\frac{1}{\left\vert Q\right\vert }%
\int_{Q}u\left( x\right)
\end{equation*}%
for certain $C>0$, thus 
\begin{equation*}
\left( M\left( u^{r}\right) \left( x\right) \right) ^{\frac{1}{r}}\leq
C.Mu\left( x\right)
\end{equation*}%
and then 
\begin{equation*}
MMu\left( x\right) \leq C.Mu\left( x\right)
\end{equation*}%
a.e. That is $Mu\in A_{1}$. We remark that this way requires two strong
results: characterization of $A_{1}$ and the reverse H\"{o}lder inequality
for $A_{p}$ weights, while proposition 1 is elementary.
\end{example}

\begin{example}
A larger class of weights that $M$ sends to $A_{1}$ are the $weak-A_{\infty
} $ weights.

We recall that $u\in A_{\infty }$ if and only if there exists positive
constants $C$ and $\delta $ such that for any cube $Q$ and any measurable $%
E\subset Q:$%
\begin{equation*}
u\left( E\right) \leq C\left( \frac{\left\vert E\right\vert }{\left\vert
Q\right\vert }\right) ^{\delta }u\left( Q\right)
\end{equation*}
\end{example}

Let's give the definition of $weak-A_{\infty }$ weights: $u\in
weak-A_{\infty }$ if and only if there exists positive constants $C$ and $%
\delta $ such that for any cube $Q$ and any measurable $E\subset Q:$%
\begin{equation}
u\left( E\right) \leq C\left( \frac{\left\vert E\right\vert }{\left\vert
Q\right\vert }\right) ^{\delta }u\left( 2Q\right)  \label{WAINFC}
\end{equation}

\begin{remark}
It's easy to prove that we can replace the factor $2$ with any constant $k>1$
obtaining an equivalent definition of $weak-A_{\infty }$.
\end{remark}

\begin{remark}
It%
\'{}%
s clear that if $u\in A_{\infty }$ then $u\in weak-A_{\infty }$ because any $%
u\in A_{\infty }$ is a doubling weight (see [D]), that is $u\left( 2Q\right)
\leq C.u\left( Q\right) $ for some $C>0$ and for every cube $Q$.
\end{remark}

It's a known result that an equivalent condition for $u$ to be in $A_{\infty
}$ is to belong to a $RHI$ class, that means that for some $r>1$ and $C>0$
it holds for any cube $Q$%
\begin{equation*}
\left( \frac{1}{\left\vert Q\right\vert }\int_{Q}u^{r}\right) ^{\frac{1}{r}%
}\leq C.\frac{1}{\left\vert Q\right\vert }\int_{Q}u
\end{equation*}

\begin{remark}
Let%
\'{}%
s remark that those weights that belongs to $weak-A_{\infty }$ but that don%
\'{}%
t belong to $A_{\infty }$ are always non-doubling weights.
\end{remark}

\begin{remark}
A corollary that we can obtain immediately taking suprema on the $RHI$
condition for $A_{\infty }$ weights is that for any $x\in \mathbb{R}^{n}$ 
\begin{equation*}
\left( M\left( u^{r}\right) \left( x\right) \right) ^{\frac{1}{r}}\leq
C.Mu\left( x\right)
\end{equation*}
\end{remark}

It can be obtained for $weak-A_{\infty }$ weights a condition analogous to $%
RHI$ as we can see in the next:

\begin{lemma}
If $u\in weak-A_{\infty }$ there are some $r>1$ and $C>0$ such that for any
cube $Q$%
\begin{equation*}
\left( \frac{1}{\left\vert Q\right\vert }\int_{Q}u^{r}\right) ^{\frac{1}{r}%
}\leq C.\frac{1}{\left\vert 2Q\right\vert }\int_{2Q}u
\end{equation*}
\end{lemma}

\begin{proof}
Let $Q$ any cube and $E_{t}=\{x\in Q:u\left( x\right) >t\}$. Now, appliying
the definition of $E_{t}\ $and (\ref{WAINFC}) we have $t.\left\vert
E_{t}\right\vert \leq u\left( E_{t}\right) \leq C.\frac{\left\vert
E_{t}\right\vert ^{\delta }}{\left\vert Q\right\vert ^{\delta }}.u\left(
2Q\right) $. Hence, using $\left\vert 2Q\right\vert =2^{n}\left\vert
Q\right\vert $ and incorporating the factor $2^{n}$ to the constant $C$:%
\begin{equation*}
t.\left\vert E_{t}\right\vert ^{1-\delta }\leq C.\left\vert Q\right\vert
^{1-\delta }.\frac{u\left( 2Q\right) }{\left\vert 2Q\right\vert }
\end{equation*}%
so 
\begin{equation*}
\left\vert E_{t}\right\vert \leq C.t^{\frac{-1}{1-\delta }}.\left\vert
Q\right\vert .\left( \frac{u\left( 2Q\right) }{\left\vert 2Q\right\vert }%
\right) ^{^{\frac{1}{1-\delta }}}
\end{equation*}

Now we use this inequality in the layer-cake formula. Let%
\'{}%
s be $k\in \left( 0,\infty \right) $ that we will chose later: 
\begin{equation*}
\int_{Q}u^{r}=\int_{0}^{\infty }rt^{r-1}\left\vert E_{t}\right\vert
dt=\int_{0}^{\infty }rt^{r-1}\left\vert E_{t}\right\vert
dt=\int_{0}^{k}rt^{r-1}\left\vert E_{t}\right\vert dt+\int_{k}^{\infty
}rt^{r-1}\left\vert E_{t}\right\vert dt
\end{equation*}%
then 
\begin{equation*}
\int_{Q}u^{r}\leq \int_{0}^{k}rt^{r-1}\left\vert Q\right\vert
dt+C\int_{k}^{\infty }rt^{r-1}t^{\frac{-1}{1-\delta }}.\left\vert
Q\right\vert .\left( \frac{u\left( 2Q\right) }{\left\vert 2Q\right\vert }%
\right) ^{^{\frac{1}{1-\delta }}}dt
\end{equation*}%
that is:%
\begin{equation*}
\int_{Q}u^{r}\leq \left\vert Q\right\vert .\left. t^{r}\right\vert
_{0}^{k}+C.\left\vert Q\right\vert .\left( \frac{u\left( 2Q\right) }{%
\left\vert 2Q\right\vert }\right) ^{^{\frac{1}{1-\delta }}}.\frac{r}{r-\frac{%
1}{1-\delta }}.\left. t^{r-\frac{1}{1-\delta }}\right\vert _{k}^{\infty }
\end{equation*}%
then, for $r:1<r<\frac{1}{1-\delta }$ we get: 
\begin{equation*}
\frac{1}{\left\vert Q\right\vert }\int_{Q}u^{r}\leq k^{r}+C.\frac{r}{\frac{1%
}{1-\delta }-r}.\left( \frac{u\left( 2Q\right) }{\left\vert 2Q\right\vert }%
\right) ^{^{\frac{1}{1-\delta }}}.k^{r-\frac{1}{1-\delta }}
\end{equation*}%
Now choosing $k=\frac{u\left( 2Q\right) }{\left\vert 2Q\right\vert }$ it
results: 
\begin{equation*}
\frac{1}{\left\vert Q\right\vert }\int_{Q}u^{r}\leq \left( \frac{u\left(
2Q\right) }{\left\vert 2Q\right\vert }\right) ^{r}+C.\frac{r}{\frac{1}{%
1-\delta }-r}.\left( \frac{u\left( 2Q\right) }{\left\vert 2Q\right\vert }%
\right) ^{^{\frac{1}{1-\delta }}}.\left( \frac{u\left( 2Q\right) }{%
\left\vert 2Q\right\vert }\right) ^{r-\frac{1}{1-\delta }}
\end{equation*}%
,hence

\begin{equation*}
\frac{1}{\left\vert Q\right\vert }\int_{Q}u^{r}\leq \left( C.\frac{r}{\frac{1%
}{1-\delta }-r}\right) \left( \frac{u\left( 2Q\right) }{\left\vert
2Q\right\vert }\right) ^{r}
\end{equation*}%
and renaming the constant we have:%
\begin{equation*}
\left( \frac{1}{\left\vert Q\right\vert }\int_{Q}u^{r}\right) ^{\frac{1}{r}%
}\leq C.\frac{u\left( 2Q\right) }{\left\vert 2Q\right\vert }
\end{equation*}
\end{proof}

\begin{corollary}
From the previous lemma it%
\'{}%
s obvious that the pointwise inequality 
\begin{equation}
\left( M\left( u^{r}\right) \left( x\right) \right) ^{\frac{1}{r}}\leq
C.Mu\left( x\right)  \label{PRHI}
\end{equation}%
still remains true for $weak-A_{\infty }$ weights and using Neugebauer's
Lemma the weights $u\in weak-A_{\infty }$ satisfy that $Mu\in A_{1}.$
\end{corollary}

\begin{remark}
Actually the condition: $\left( \frac{1}{\left\vert Q\right\vert }%
\int_{Q}u^{r}\right) ^{\frac{1}{r}}\leq C.\frac{1}{\left\vert 2Q\right\vert }%
\int_{2Q}u$ characterizes the $weak-A_{\infty }$ weights; it can be proved
that the converse of the previous lemma is also true, nevertheless we will
not need here that result. As we mentioned in a previous remark we can
replace the constant $2$ for any $k>1$, so $u\in weak-A_{\infty }$ iff there
exists some positive constant $C$ such that for any $k>1$ and every cube $Q$%
\begin{equation}
\left( \frac{1}{\left\vert Q\right\vert }\int_{Q}u^{r}\right) ^{\frac{1}{r}%
}\leq C.\frac{1}{\left\vert kQ\right\vert }\int_{kQ}u  \label{WAINFLRHI}
\end{equation}
\end{remark}

\begin{remark}
We have already seen that $A_{\infty }\subset weak-A_{\infty }\subset
M^{-1}\left( A_{\infty }\right) $ where we denote $M^{-1}\left( A_{\infty
}\right) $ the class of weights $u$ such that $Mu\in A_{\infty }$.

It%
\'{}%
s interesting to observe that this question has a close relationship with
another one involving the weighted Fefferman-Stein inequality in $%
L^{p}\left( w\right) :$ 
\begin{equation}
\left\Vert f\right\Vert _{L^{p}\left( w\right) }\leq c\left\Vert
f^{\#}\right\Vert _{L^{p}\left( w\right) }\qquad \left( 1<p<\infty \right)
\label{FFS}
\end{equation}

for some $c>0$, and for every $f\in L^{p}\ $such that $f\in S_{0}\left( 
\mathbb{R}^{n}\right) $, where $S_{0}\left( \mathbb{R}^{n}\right) \ $is the
space of measurable functions $f$ on $\mathbb{R}^{n}$ such that for any $t>0$
\begin{equation*}
\mu _{f}\left( t\right) =\left\vert \left\{ x\in \mathbb{R}^{n}:\left\vert
f\left( x\right) \right\vert >t\right\} \right\vert <\infty
\end{equation*}

The inequality \ref{FFS} is equivalent to many interesting others, for
instance, with the same hypothesis of \ref{FFS}:%
\begin{equation*}
\left\Vert Mf\right\Vert _{L^{p}\left( w\right) }\leq c\left\Vert
f^{\#}\right\Vert _{L^{p}\left( w\right) }\qquad \left( 1<p<\infty \right)
\end{equation*}%
or for some $c>0,$ $r>1$ and for any $f\in L_{loc}^{1}\left( \mathbb{R}%
^{n}\right) $%
\begin{equation}
\int_{\mathbb{R}^{n}}\mathcal{M}_{p,r}\left( f,w\right) \left\vert
f\right\vert dx\leq c\int_{\mathbb{R}^{n}}\left( Mf\right) ^{p}wdx\qquad
\left( 1<p<\infty \right)  \label{RFS}
\end{equation}%
where $\mathcal{M}_{p,r}\left( f,w\right) =\sup\limits_{Q\ni x}\left( \frac{1%
}{\left\vert Q\right\vert }\int_{Q}\left\vert f\right\vert \right)
^{p-1}\left( \frac{1}{\left\vert Q\right\vert }\int_{Q}w^{r}\right) ^{\frac{1%
}{r}}$. The equivalence of those inequalities is proven in [L3].

Related to the -at our knowledge- open question about for which weights the
former inequalities hold are the following inclusions of nested classes: $%
A_{\infty }\subset weak-A_{\infty }\subset C_{p+\varepsilon }\subset C_{p}$
where $\varepsilon >0$ and $C_{p}$ condition means that there exists $%
c,\delta >0$ such that for any cube $Q$ and any measurable $E\subset Q$ 
\begin{equation*}
u\left( E\right) \leq c\left( \frac{\left\vert E\right\vert }{\left\vert
Q\right\vert }\right) ^{\delta }\int_{\mathbb{R}^{n}}\left( M\chi
_{Q}\right) ^{p}u
\end{equation*}%
Remember that for $u\in A_{1}$for any cube $Q$ and any measurable $E\subset
Q $%
\begin{equation*}
u\left( E\right) \leq c\left( \frac{\left\vert E\right\vert }{\left\vert
Q\right\vert }\right) ^{\delta }u\left( Q\right) =c\left( \frac{\left\vert
E\right\vert }{\left\vert Q\right\vert }\right) ^{\delta }\int_{\mathbb{R}%
^{n}}\left( \chi _{Q}\right) ^{p}u
\end{equation*}

and for $weak-A_{\infty }$ weights: $u\in weak-A_{\infty }$ if and only if
there exists positive constants $C$ and $\delta $ such that for any cube $Q$
and any measurable $E\subset Q:$%
\begin{equation*}
u\left( E\right) \leq C\left( \frac{\left\vert E\right\vert }{\left\vert
Q\right\vert }\right) ^{\delta }\int_{\mathbb{R}^{n}}\left( \chi
_{2Q}\right) ^{p}u
\end{equation*}%
and the mentioned inclusion are obvious. It can be found in [L3] (see also
[Y]) that $C_{p}$ is necessary and $C_{p+\varepsilon }$ is sufficient for %
\ref{RFS} or \ref{FFS} -and in [L3] is introduced a new sufficient condition 
$\widetilde{C_{p}}$ instead of $C_{p+\varepsilon }$ but it is not known if $%
\widetilde{C_{p}}$ or $C_{p+\varepsilon }$ are necessary conditions.

The inclusion relations from $A_{\infty }\subset weak-A_{\infty }\subset
M^{-1}\left( A_{\infty }\right) $ and $A_{\infty }\subset weak-A_{\infty
}\subset C_{p+\varepsilon }\subset C_{p}$ and the former inequalities sems
to be close linked: For instance $u\in C_{p}$ is necesssary for \ref{RFS},
and \ref{RFS} implies that for any $Q$ we have that $\left( \frac{1}{%
\left\vert Q\right\vert }\int_{Q}u^{r}\right) ^{\frac{1}{r}}\leq c\frac{1}{%
\left\vert Q\right\vert }\int_{\mathbb{R}^{n}}\left( M\chi _{Q}\right) ^{p}u$%
, which is a bit weaker than $\left( \frac{1}{\left\vert Q\right\vert }%
\int_{Q}u^{r}\right) ^{\frac{1}{r}}\leq C.\frac{1}{\left\vert Q\right\vert }%
\int_{\mathbb{R}^{n}}\left( \chi _{2Q}\right) ^{p}u$ that it is equivalent
to $weak-A_{\infty }$.

Additionally in [L3] is proven that $C_{p}$ is necessary for $\int_{\mathbb{R%
}^{n}}\mathcal{M}_{p,r}\left( f,w\right) \left\vert f\right\vert dx\leq
c\int_{\mathbb{R}^{n}}\left( Mf\right) ^{p}wdx$, that is \ref{RFS}$\ $%
implies $C_{p}$.

On the other hand, using the lemma of Neugebauer telling us $\left(
Mu^{r}\right) ^{\frac{1}{r}}\left( x\right) \leq C.Mu\left( x\right) $ for $%
u\in M^{-1}\left( A_{\infty }\right) $ for some $C>0,r>1$ and the definition
of $\mathcal{M}_{p,r}\left( f,u\right) $ we obtain that if $u\in
M^{-1}\left( A_{\infty }\right) $ then 
\begin{equation*}
\mathcal{M}_{p,r}\left( f,w\right) \left( u\right) =\sup\limits_{Q\ni
x}\left( \frac{1}{\left\vert Q\right\vert }\int_{Q}\left\vert f\right\vert
\right) ^{p-1}\left( \frac{1}{\left\vert Q\right\vert }\int_{Q}u^{r}\right)
^{\frac{1}{r}}
\end{equation*}%
\begin{equation*}
\leq \sup\limits_{Q\ni x}\left( \frac{1}{\left\vert Q\right\vert }%
\int_{Q}\left\vert f\right\vert \right) ^{p-1}M_{r}u\left( x\right) \leq
\end{equation*}%
\begin{equation*}
\sup\limits_{Q\ni x}\left( \frac{1}{\left\vert Q\right\vert }%
\int_{Q}\left\vert f\right\vert \right) ^{p-1}.C.Mu\left( x\right) \leq
\left( Mf\right) ^{p-1}\left( x\right) .C.Mu\left( x\right)
\end{equation*}%
and then integrating we have: 
\begin{equation}
\int_{\mathbb{R}^{n}}\mathcal{M}_{p,r}\left( f,w\right) \left\vert
f\right\vert dx\leq c\int_{\mathbb{R}^{n}}\left( Mf\right) ^{p}Mwdx
\label{TRFS}
\end{equation}%
(compare with \ref{RFS}). So we have that $M^{-1}\left( A_{\infty }\right) $
implies \ref{TRFS} and \ref{RFS} implies $C_{p}$.
\end{remark}

\bigskip

A\ PAIR\ OF\ APPLICATIONS

\begin{application}
Using the criterion: $Mu\in A_{\infty }$ if and only if for any $\lambda \in
\left( 0,1\right) $ it holds that $m_{\lambda }\left( Mu\right) \thickapprox
M\left( Mu\right) $ we can derive from this result a characterization of the 
$A_{1}$ weights similar to the construction of Coifman and Rochberg.
\end{application}

First of all we introduce the definition of the local sharp maximal
operator; for $0<\lambda <1$ we define:

$M_{\lambda }^{\#}f\left( x\right) =\sup\limits_{Q\ni
x}\inf\limits_{c}\left( \left( f-c\right) \chi _{Q}\right) ^{\ast }\left(
\lambda \left\vert Q\right\vert \right) $

The sharp maximal function have a role quite similar to the Hardy-Littlewood
maximal operator for the local sharp maximal functions because there are
positive constants $c_{1}$ and $c_{2}$ such that for $f\in L_{loc}^{1}$:

\begin{equation*}
c_{1}MM_{\lambda }^{\#}f\left( x\right) \leq f^{\#}\left( x\right) \leq
c_{2}MM_{\lambda }^{\#}f\left( x\right)
\end{equation*}%
(see [J-T]). Using the former inequalities we easily get that for the sharp
function an statement similar to the first one of the Coifman-Rochberg
theorem:

\begin{lemma}
Let $f\in L_{loc}^{1}\left( \mathbb{R}^{n}\right) $ and $0\leq \delta <1$,
then $w\left( x\right) =f^{\#}\left( x\right) ^{\delta }\ $is in $A_{1}$.
\end{lemma}

\begin{proof}
For $c_{1}^{\delta }\left( MM_{\lambda }^{\#}f\left( x\right) \right)
^{\delta }\leq f^{\#}\left( x\right) ^{\delta }\leq c_{2}^{\delta }\left(
MM_{\lambda }^{\#}f\left( x\right) \right) ^{\delta }$ and $\left(
MM_{\lambda }^{\#}f\left( x\right) \right) ^{\delta }\in A_{1}$ because the
mentioned result of Coifman and Rochberg. Now $Mf^{\#}\left( x\right)
^{\delta }\leq M\left( c_{2}^{\delta }\left( MM_{\lambda }^{\#}f\left(
x\right) \right) ^{\delta }\right) \leq c_{2}^{\delta }[MM_{\lambda
}^{\#}f\left( x\right) ]_{A_{1}}\left( MM_{\lambda }^{\#}f\left( x\right)
\right) ^{\delta }\leq Cf^{\#}\left( x\right) ^{\delta }$ with constant $C=%
\frac{c_{2}^{\delta }}{c_{1}^{\delta }}[MM_{\lambda }^{\#}f\left( x\right)
]_{A_{1}}$, so $f^{\#}\left( x\right) ^{\delta }\in A_{1}$.
\end{proof}

We don%
\'{}%
t know if any $w\in A_{1}$ always could be written as $k\left( x\right)
f^{\#}\left( x\right) ^{\delta }$ for suitable $f\in L_{loc}^{1}$; $0<\delta
<1$ and $k,k^{-1}\in L^{\infty }$, but we can obtain a result similar to the
second part of Coifman-Rochberg theorem if we added a multiple of the local
maximal function $m_{\lambda }$:

\begin{proposition}
If $w\in A_{1}$ then there are $k\left( x\right) $ such that $k,k^{-1}\in
L^{\infty }$ and a constants $c,d>0$ such that $w\left( x\right) =k\left(
x\right) \left( c.\left( \left( w^{\alpha }\left( x\right) \right)
^{\#}\right) ^{\delta }+d.\left( m_{\lambda }w^{\alpha }\left( x\right)
\right) ^{\delta }\right) $
\end{proposition}

\begin{proof}
If $w\in A_{1}$ we can use the property E) to take $\alpha >1$ such that $%
w^{\alpha }\in A_{1}$. Thus $M\left( w^{\alpha }\right) \in A_{1}$. Now
using for $w^{\alpha }$ the above criterion that establishes that $Mu\in
A_{1}$ if and only if $m_{\lambda }\left( Mu\right) \thickapprox M\left(
Mu\right) $ and then in such situation: $m_{\lambda }\left( M\left(
w^{\alpha }\right) \right) \thickapprox M\left( M\left( w^{\alpha }\right)
\right) \thickapprox M\left( w^{\alpha }\right) \thickapprox w^{\alpha }$,
also we have that $Mw\thickapprox w$ because $w\in A_{1}$ and also using the
pointwise inequalities mentioned in xi) and vii): $m_{\lambda }\left(
Mu\right) \left( x\right) \leq c_{\lambda ,n}.u^{\#}\left( x\right)
+Mu\left( x\right) $ and $m_{\lambda }\left( Mu\right) \left( x\right) \leq
c_{\lambda ,n}.u^{\#}\left( x\right) +Mu\left( x\right) $, for $u=w^{\alpha
} $ we have:%
\begin{equation*}
w\left( x\right) ^{\alpha }\leq M\left( w^{\alpha }\right) \left( x\right)
\leq c_{\lambda ,n}\cdot \left( w^{\alpha }\right) ^{\#}\left( x\right)
+m_{\lambda }\left( w^{\alpha }\right) \left( x\right)
\end{equation*}%
Then with $\delta =\frac{1}{\alpha }$ it is $0<\delta <1\ $and $\alpha
\delta =1.$ Also we will use property (i): $u^{\#}\leq 2Mu$ pointwise,
properties (vi) ($\left\vert f\left( x\right) \right\vert \leq m_{\lambda
}f\left( x\right) $) and (x) ($f\left( x\right) \leq Mf\left( x\right) $)
and that if $f\left( x\right) \leq g\left( x\right) $ a.e. for positive
functions then $Mf\left( x\right) \leq Mg\left( x\right) $ and $m_{\lambda
}\left( f\right) \left( x\right) \leq m_{\lambda }\left( g\right) \left(
x\right) $ a.e.

Further we use the sublinearity of $M$ and the facts that $w^{\alpha }$ and $%
w$ are in $A_{1}$ and then because the criterion, we can use that for $w\in
A_{1}$ then $Mw\in A_{1}$ too and it occurs that $m_{\lambda }\left(
Mw\right) \thickapprox M\left( Mw\right) \thickapprox Mw\thickapprox w$. We
will number or rename the constants that appear. Also we will use that $%
M\left( \left( Mw^{\alpha }\right) ^{\delta }\right) \leq C\left( Mw^{\alpha
}\right) ^{\delta }$ (because $\left( Mf\right) ^{\delta }\in A_{1}$ by
Coifman-Rochberg). So we get: 
\begin{equation*}
w\left( x\right) \leq \left( c_{1}\cdot \left( w^{\alpha }\right)
^{\#}\left( x\right) +m_{\lambda }\left( w^{\alpha }\right) \left( x\right)
\right) ^{\delta }
\end{equation*}%
\begin{equation*}
\leq c_{2}\cdot \left( \left( w^{\alpha }\right) ^{\#}\left( x\right)
\right) ^{\delta }+\left( m_{\lambda }\left( w^{\alpha }\right) \left(
x\right) \right) ^{\delta }
\end{equation*}%
\begin{equation*}
\leq M\left( c_{2}\cdot \left( \left( w^{\alpha }\right) ^{\#}\left(
x\right) \right) ^{\delta }+\left( m_{\lambda }\left( w^{\alpha }\right)
\left( x\right) \right) ^{\delta }\right)
\end{equation*}%
\begin{equation*}
\leq c_{2}\cdot M\left( \left( \left( w^{\alpha }\right) ^{\#}\left(
x\right) \right) ^{\delta }\right) +M\left( \left( m_{\lambda }\left(
w^{\alpha }\right) \left( x\right) \right) ^{\delta }\right)
\end{equation*}%
\begin{equation*}
\leq c_{2}\cdot M\left( 2^{\delta }M\left( w^{\alpha }\right) \left(
x\right) ^{\delta }\right) +M\left( \left( m_{\lambda }\left( Mw^{\alpha
}\right) \left( x\right) \right) ^{\delta }\right)
\end{equation*}%
\begin{equation*}
\leq c_{3}M\left( M\left( w^{\alpha }\right) \left( x\right) ^{\delta
}\right) +M\left( \left( m_{\lambda }\left( Mw^{\alpha }\right) \left(
x\right) \right) ^{\delta }\right)
\end{equation*}%
\begin{equation*}
\leq c_{3}M\left( c_{4}\left( w^{\alpha }\right) \left( x\right) ^{\delta
}\right) +M\left( \left( c_{5}w\left( x\right) ^{\alpha }\right) ^{\delta
}\right)
\end{equation*}%
\begin{equation*}
\leq c_{6}Mw\left( x\right) +c_{7}Mw\left( x\right) =c_{8}Mw\left( x\right)
\leq Cw\left( x\right)
\end{equation*}%
Thus we obtain: 
\begin{equation*}
w\left( x\right) \leq c_{1}^{\delta }\cdot \left( \left( w^{\alpha }\right)
^{\#}\left( x\right) \right) ^{\delta }+\left( m_{\lambda }\left( w^{\alpha
}\right) \left( x\right) \right) ^{\delta }\leq Cw\left( x\right)
\end{equation*}%
and then $k\left( x\right) =\frac{w\left( x\right) }{c_{2}.\left( \left(
w^{\alpha }\left( x\right) \right) ^{\#}\right) ^{\delta }+\left( m_{\lambda
}w^{\alpha }\left( x\right) \right) ^{\delta }}$ satisfy that $k\in
L^{\infty }$ and $k^{-1}\in L^{\infty }$

So $w\left( x\right) =k\left( x\right) \left( c.\left( \left( w^{\alpha
}\left( x\right) \right) ^{\#}\right) ^{\delta }+d.\left( m_{\lambda
}w^{\alpha }\left( x\right) \right) ^{\delta }\right) $ with $k,k^{-1}\in
L^{\infty }$ and $\delta \in \left( 0,1\right) $ for $c=c_{2}$ and $d=1$.
\end{proof}

On the other hand we have:

\begin{lemma}
If $0<\delta <1$ and $u\in A_{1}$ then $\left( m_{\lambda }u\left( x\right)
\right) ^{\delta }\in A_{1}$
\end{lemma}

\begin{proof}
Using that $u\in A_{1}$, then $Mu\in A_{1}$ and $m_{\lambda }\left(
Mu\right) \thickapprox M\left( Mu\right) \thickapprox Mu\thickapprox u$ and
that $\left( MMu\right) ^{\delta }\in A_{1}$ (by Coifman-Rochberg theorem)
we have the following inequalities -with multiplicative constants that we
will be renumbering -:$\ M\left( \left( m_{\lambda }u\right) ^{\delta
}\right) \leq M\left( \left( m_{\lambda }Mu\right) ^{\delta }\right) \leq
M\left( \left( C_{1}MMu\right) ^{\delta }\right) =C_{2}M\left( \left(
MMu\right) ^{\delta }\right) \leq C_{3}\left( MMu\right) ^{\delta }\leq
C_{4}\left( m_{\lambda }\left( Mu\right) \right) ^{\delta }\leq C_{5}\left(
m_{\lambda }\left( C_{4}u\right) \right) ^{\delta }=C_{5}\left( m_{\lambda
}u\right) ^{\delta }$ and then we get that $\left( m_{\lambda }u\right)
^{\delta }\in A_{1}$.
\end{proof}

\begin{remark}
It%
\'{}%
s elementary that if $v_{1},v_{2}$ are non-negative functions with $%
v_{1},v_{2}\in A_{1}$ and if $c$ and $d$ are non-negative constans then $%
cv_{1}+dv_{2}\in A_{1}$ and $[cv_{1}+dv_{2}]_{A_{1}}\leq \max \left\{
[v_{1}]_{A_{1}},[v_{2}]_{A_{2}}\right\} $.
\end{remark}

Compiling the last two lemmas, the proposition and the previous remark we
have a theorem similar to the Coifman-Rochberg result:

\begin{theorem}
(1) If $0<\delta <1$, $f\in L_{loc}^{1}\left( \mathbb{R}^{n}\right) $ and $%
u\in A_{1}$ and $c,d$ non-negative constants then $\left( c\cdot
f^{\#}\left( x\right) +d\cdot m_{\lambda }u\left( x\right) \right) ^{\delta
}\in A_{1}$.

(2) Conversely, if $w\in A_{1}$ then there are $f\in L_{loc}^{1}\left( 
\mathbb{R}^{n}\right) $, $u\in A_{1},$ non-negative constants $c\ $and $d$,
and $k\left( x\right) $with $k,k^{-1}\in L^{\infty }$ such that $w\left(
x\right) =k\left( x\right) \left( c.f^{\#}\left( x\right) ^{\delta
}+d.m_{\lambda }u\left( x\right) ^{\delta }\right) $.
\end{theorem}

\begin{proof}
The first statement is consequence of the latter remark and from the lemmas
telling us that $f^{\#}\left( x\right) ^{\delta }\ \ $and $\left( m_{\lambda
}u\left( x\right) \right) ^{\delta }$ are in $A_{1}$ for $f\in L_{loc}^{1}$
and $u\in A_{1}$.

The second was obtained in the latter proposition for $f=u=w^{\alpha }$
taking a suitable $\alpha >1$ such that $w^{\alpha }\in A_{1}$. The
existence of that $\alpha $ is guaranteed by property E.
\end{proof}

\begin{remark}
The previous result, like the Coifman-Rochberg Theorem, presents a class of
functions, included in $A_{1}$, such that any $A_{1}$ weight differs from
some element of that class only by a factor function $k\left( x\right) $
that it is bounded and bounded away from zero, that is $k,k^{-1}\in
L^{\infty }$. Another remarkable example is given by the functions in the
image of an operator obtained by means of a variant of the Rubio de Francia
algorithm.

The usual construction (see for instance [G2]) involves some sublinear
operator bounded in $L^{p}\left( \mu \right) $ with $p\geq 1$ for certain
measure $\mu $ and it is defined for $f\in L^{p}\left( \mu \right) $ by: 
\begin{equation*}
Rf\left( x\right) =\sum\limits_{k=0}^{\infty }\frac{T^{k}f\left( x\right) }{%
\left( 2\left\Vert T\right\Vert _{p,\mu }\right) ^{k}}
\end{equation*}%
where $T^{0}$ is the identity and $T^{k}=T\circ T\circ ...\circ T,$ k times.
Some basic properties of $R$ are:

i) $f\left( x\right) \leq Rf\left( x\right) $ a.e.

ii) $\left\Vert Rf\right\Vert _{p,\mu }\leq 2\left\Vert f\right\Vert _{p,\mu
}$

iii) $T\left( Rf\right) \left( x\right) \leq 2\left\Vert T\right\Vert
_{p,\mu }Rf\left( x\right) $ a.e.

For $T=M,$ the Hardy-Littlewood maximal operator and the usual Lebesgue
measure in $\mathbb{R}^{n}$ the third property means $M\left( Rf\right)
\left( x\right) \leq 2\left\Vert M\right\Vert _{p}Rf\left( x\right) $ thus $%
Rf\in A_{1}$ for any $f\in L^{p}$ with $[Rf]_{A_{1}}\leq 2\left\Vert
M\right\Vert _{p}$. For to characterize the whole $A_{1}$ might be necessary
to change this procedure for to avoid the issue about the belonging to $%
L^{p} $ (for instance if $f\in L^{1}$ then $Mf$ is never in $L^{1}$ except
when $f$ is identicaly $0$). Notwithstanding we can give the following:
\end{remark}

\begin{proposition}
$u\in A_{1}$ if and only if there are $C>0,$ $f\in L_{loc}^{1}$ and $k\left(
x\right) $ with $k,k^{-1}\in L^{\infty }$ such that $w\left( x\right)
=\sum\limits_{k=0}^{\infty }\frac{M^{k}f\left( x\right) }{C^{k}}$ is well
defined, $w\in A_{1}$ and $u\left( x\right) =k\left( x\right) \cdot w\left(
x\right) $.
\end{proposition}

\begin{proof}
The proof is almost trivial. The "if" part is immediate because if $u\left(
x\right) =k\left( x\right) \cdot w\left( x\right) $ with $w\in A_{1}$ and $%
k,k^{-1}\in L^{\infty }$ then 
\begin{equation*}
Mu\left( x\right) \leq \left\Vert k\right\Vert _{\infty }M\left( w\right)
\left( x\right) \leq \left\Vert k\right\Vert _{\infty }[w]_{A_{1}}w\left(
x\right)
\end{equation*}%
\begin{equation*}
\leq \left\Vert k\right\Vert _{\infty }[w]_{A_{1}}\left\Vert
k^{-1}\right\Vert _{\infty }k\left( x\right) \cdot w\left( x\right) \leq
\lbrack w]_{A_{1}}\left\Vert k\right\Vert _{\infty }\left\Vert
k^{-1}\right\Vert _{\infty }u\left( x\right)
\end{equation*}%
that is $u\in A_{1}$ and $[u]_{A_{1}}\leq \lbrack w]_{A_{1}}\left\Vert
k\right\Vert _{\infty }\left\Vert k^{-1}\right\Vert _{\infty }$.

For the "only if" part let%
\'{}%
s take $f=u\in L_{loc}^{1}$ (because u is a weight), $C=2[u]_{A_{1}}$, $%
w\left( x\right) =Ru\left( x\right) :=\sum\limits_{k=0}^{\infty }\frac{%
M^{k}u\left( x\right) }{\left( 2[u]_{A_{1}}\right) ^{k}}$ and $k\left(
x\right) =\frac{u\left( x\right) }{Ru\left( x\right) }$.

Iterating we have that $M^{k}u\left( x\right) \leq \lbrack
u]_{A_{1}}^{k}u\left( x\right) $ a.e. and then 
\begin{equation*}
0\leq \frac{M^{k}u\left( x\right) }{\left( 2[u]_{A_{1}}\right) ^{k}}\leq 
\frac{u\left( x\right) [u]_{A_{1}}^{k}}{\left( 2[u]_{A_{1}}\right) ^{k}}=%
\frac{u\left( x\right) }{2^{k}}\text{ a.e.}
\end{equation*}
Thus $\sum\limits_{k=0}^{\infty }\frac{M^{k}u\left( x\right) }{\left(
2[u]_{A_{1}}\right) ^{k}}$ is convergent a.e., $w\left( x\right) =Ru\left(
x\right) =\sum\limits_{k=0}^{\infty }\frac{M^{k}u\left( x\right) }{\left(
2[u]_{A_{1}}\right) ^{k}}$ is well defined and $Mw\left( x\right) \leq
\sum\limits_{k=0}^{\infty }\frac{M^{k+1}u\left( x\right) }{\left(
2[u]_{A_{1}}\right) ^{k}}\leq \sum\limits_{k=0}^{\infty }\frac{%
[u]_{A_{1}}M^{k}u\left( x\right) }{\left( 2[u]_{A_{1}}\right) ^{k}}%
=[u]_{A_{1}}w\left( x\right) $, that is $w\in A_{1}$ and $[w]_{A_{1}}\leq
\lbrack u]_{A_{1}}$.

Finally 
\begin{equation*}
u\left( x\right) \leq w\left( x\right) =\sum\limits_{k=0}^{\infty }\frac{%
M^{k}u\left( x\right) }{\left( 2[u]_{A_{1}}\right) ^{k}}\leq
\sum\limits_{k=0}^{\infty }\frac{[u]_{A_{1}}^{k}u\left( x\right) }{\left(
2[u]_{A_{1}}\right) ^{k}}=u\left( x\right) \sum\limits_{k=0}^{\infty }\frac{1%
}{2^{k}}=2u\left( x\right)
\end{equation*}%
so $1\leq \frac{w\left( x\right) }{u\left( x\right) }\leq 2$ and then $%
k\left( x\right) =\frac{u\left( x\right) }{w\left( x\right) }$ satisfies
that $k,k^{-1}\in L^{\infty }$ with $\left\Vert k\right\Vert _{\infty }\leq
1 $ and $\left\Vert k^{-1}\right\Vert _{\infty }\leq 2$, thus $u\left(
x\right) =k\left( x\right) \cdot w\left( x\right) $ with $%
w=\sum\limits_{k=0}^{\infty }\frac{M^{k}u\left( x\right) }{\left(
2[u]_{A_{1}}\right) ^{k}}\in A_{1}$ and $k,k^{-1}\in L^{\infty }$.
\end{proof}

\begin{application}
For those weights $u$ such that $Mu\in A_{\infty }$ and hence $Mu\in A_{1}$
we can improve some known inequalities for singular integral operators. For
instance if $T$ is a Calder\'{o}n-Zygmund singular integral operator (see
[G] for a definition) the following weighted inequalities were proved for $%
1<p<\infty $ by C. P\'{e}rez ([P]) -previously J.M. Wilson obtained the
first inequality for $1<p<2$-:%
\begin{equation*}
\int_{\mathbb{R}^{n}}\left\vert Tf\right\vert ^{p}u\leq C_{p}\int_{\mathbb{R}%
^{n}}\left\vert f\right\vert ^{p}M^{[p]+1}u
\end{equation*}%
and then%
\begin{equation*}
u\left( \{x\in \mathbb{R}^{n}:\left\vert Tf\left( x\right) \right\vert
>\lambda \}\right) \leq \frac{C_{p}}{\lambda ^{p}}\int_{\mathbb{R}%
^{n}}\left\vert f\right\vert ^{p}M^{[p]+1}u
\end{equation*}%
the last one for the case $p=1$ looks:%
\begin{equation*}
u\left( \{x\in \mathbb{R}^{n}:\left\vert Tf\left( x\right) \right\vert
>\lambda \}\right) \leq \frac{C_{2}}{\lambda }\int_{\mathbb{R}%
^{n}}\left\vert f\right\vert M^{2}u
\end{equation*}%
where $[p]$ is the integer part of $p$ and $M^{k}$ is the $k-th$ iterate
composition of $M$. The strong inequality is sharp in the sense that $[p]+1$
cannot be replaced by $[p]$, and the weak case is sharp when $p$ is not an
integer and it is an open question -at our knowledge- if it is possible to
replace $M^{[p]+1}$ with $M^{[p]}$ if $p\in \mathbb{N}$ -and $M^{2}$ with $M$
in the last inequality-.

Now for a weight $u$ such that $Mu\in A_{\infty }$ we have that actually $%
Mu\in A_{1}$ and then there are a constant $C>0$ such that for almost every $%
x\in \mathbb{R}^{n}:$ $M^{2}u\left( x\right) \leq C.Mu\left( x\right) $, and
using that if in almost everywhere $f\left( x\right) \leq g\left( x\right) $
then $Mf\left( x\right) \leq Mg\left( x\right) ,$ we can iterate in $%
M^{2}u\left( x\right) \leq C.Mu\left( x\right) $ to obtain $M^{k}u\left(
x\right) \leq C^{k}.Mu\left( x\right) $, then with $C=C_{p}^{k}$ we have for
the Calder\'{o}n-Zygmund singular integral operators and the weights $u$
with $Mu\in A_{\infty }$:%
\begin{equation*}
\int_{\mathbb{R}^{n}}\left\vert Tf\right\vert ^{p}u\leq C\int_{\mathbb{R}%
^{n}}\left\vert f\right\vert ^{p}Mu
\end{equation*}%
\begin{equation*}
u\left( \{x\in \mathbb{R}^{n}:\left\vert Tf\left( x\right) \right\vert
>\lambda \}\right) \leq \frac{C}{\lambda ^{p}}\int_{\mathbb{R}%
^{n}}\left\vert f\right\vert ^{p}Mu
\end{equation*}%
for any $1<p<\infty $.
\end{application}

\bigskip

\bigskip

REFERENCES

[BS] Bennett C and Robert Sharpley R, Interpolation of Operators. Pure and
Applied Mathematics Series, Vol 129, 1988

[C-R] Coifman, R., Rochberg, R. Another characterization of B.M.O., Proc.
Amer. Math. Soc. 1980

[CU] Cruz-Uribe, D. Cruz-Uribe, SFO, Piecewise monotonic doubling measures,
Rocky Mtn. J. Math. 26 (1996),1-39.

[CU-P] D. Cruz-Uribe, SFO, P\'{e}rez C. Two weight extrapolation via the
maximal operator, Journal of Functional Analysis, 2000

[D] Duoandikoetxea, J. Fourier Ananlysis, Graduate studies in Mathematics,
AMS 2001

[DMO] Duoandikoetxea J, Mart\'{\i}n-Reyes F, Ombrosi S., On the $A_{\infty }$
conditions for general bases, Mathematische Zeitschrift, 2016

[G] Grafakos, Classical Fourier Analysis, Graduate studies in Mathematics,
Springer, 2000

[G2] Grafakos, Modern Fourier Analysis, Graduate studies in Mathematics,
Springer, 2009

[G-R] Garcia-Cuerva, J., and Rubio de Francia, J. L. Weighted Norm
Inequalities and Related Topics, North Holland, New York, 1985.

[J-T] B. Jawerth, A. Torchinsky, Local sharp maximal functions, J. Approx.
Theory 43 (1985) 231--270

[L] Lerner, A, On some pointwise estimates for maximal and singular integral
operators, Studia Math 138 (2000)

[L2] Lerner, A. On some pointwise inequalities, J. Math. Anal. Appl. 289
(2004)

[L3] Lerner, A. Some remarks on the Fefferman-Stein inequality, Journal
d'Analyse Math\'{e}matique October 2010, Volume 112, Issue 1, pp 329--349

[P] P\'{e}rez C. Weighted norm inequalities for singular integral operators.
C. P\'{e}rez. Journal of the London mathematical society 49, 1994

[S] E. Stein, Harmonic analysis real-variable methods, orthogonality, and
oscillatory integrals, Princeton Univ. Press, Princeton, NJ, 1993

[Y] Yabuta K., Sharp maximal function and Cp condition, Arch. Math. 55
(1990), 151--155.

\end{document}